\documentclass{amsart}

\usepackage{amsmath,amssymb,mathtools}
\usepackage{microtype}
\usepackage{needspace}
\usepackage[hidelinks]{hyperref}
\hypersetup{pdftitle={Maximal shifts below the Taylor bound},
  pdfauthor={Abed Abedelfatah},
  pdfsubject={Maximal shifts of monomial ideals}}

\numberwithin{equation}{section}
\allowdisplaybreaks

\newtheorem{theorem}{Theorem}[section]
\newtheorem{corollary}[theorem]{Corollary}
\newtheorem{lemma}[theorem]{Lemma}
\newtheorem{proposition}[theorem]{Proposition}
\theoremstyle{remark}

\newtheorem{example}[theorem]{Example}
\newtheorem{question}[theorem]{Question}

\newcommand{\kk}{K}
\newcommand{\pd}{\operatorname{pd}}
\newcommand{\reg}{\operatorname{reg}}
\newcommand{\NN}{\mathcal{N}}
\newcommand{\Dcal}{\mathcal{D}}

\title[Maximal shifts below the Taylor bound]{Maximal shifts below the Taylor bound}

\author{Abed Abedelfatah}
\address{Department of Mathematics, Braude College of Engineering, 2161002 Karmiel, Israel}
\email{abed@braude.ac.il}


\begin{document}

\begin{abstract}
Let $I\subseteq S$ be a monomial ideal with minimal generator degrees
between $e$ and $d$, where $2\leq e\leq d$.  We prove that if
$t_a(S/I)<ea$, then

$$
 t_{a+b}(S/I)\leq t_a(S/I)+
 \left\lceil\frac{(2d-1)b}{2}\right\rceil
$$

for $a\geq2$, $b\geq0$, and $a+b\leq\pd_S(S/I)$.
Consequently,

$$
 t_{a+b}(S/I)\leq ae+db-1-\left\lfloor\frac b2\right\rfloor.
$$

In particular, if $t_2(S/I)<2e$, then

$$
 t_i(S/I)\leq
 \left\lceil\frac{(2d-1)i}{2}\right\rceil-2(d-e)
$$

for $2\leq i\leq\pd_S(S/I)$.
We also obtain a regularity bound and give a quadratic family in which
our estimate for the last shift is smaller than every bound obtained
from ordinary subadditivity.  The constant $2d-1$ is sharp for $d=2$.
\end{abstract}

\maketitle

\section{Introduction}

Let $S=\kk[x_1,\ldots,x_n]$ be a standard graded polynomial ring over a
field $\kk$, and set $R=S/I$ and $p=\pd_S R$.  The maximal shifts are
\[
 t_i(R)=\max\{j:\beta^S_{i,j}(R)\neq0\}\qquad(0\leq i\leq p),
\]
where $\beta^S_{i,j}(R)$ are the graded Betti numbers and $t_0(R)=0$.
If the minimal generators of a monomial ideal have degree at most $d$,
the Taylor resolution gives
\begin{equation}\label{eq:Taylor-bound-intro}
 t_i(R)\leq di.
\end{equation}

Suppose now that the generator degrees lie between $e$ and $d$.
We study how the inequality $t_a<ea$ affects the later maximal shifts.
The next theorem gives a bound for all later shifts.

\begin{theorem}\label{thm:main}
Let $I\subseteq S$ be a monomial ideal with minimal generator degrees
between $e$ and $d$, where $2\leq e\leq d$.  If $t_a(S/I)<ea$ for some
$2\leq a\leq\pd_S(S/I)$, then
\begin{equation}\label{eq:main-intro}
 t_{a+b}(S/I)\leq t_a(S/I)+
 \left\lceil\frac{(2d-1)b}{2}\right\rceil
\end{equation}
for every $b\geq0$ with $a+b\leq\pd_S(S/I)$.
\end{theorem}

Since $t_a\leq ae-1$, Theorem~\ref{thm:main} gives
\begin{equation}\label{eq:mixed-intro}
 t_{a+b}\leq ae+db-1-\left\lfloor\frac b2\right\rfloor.
\end{equation}
Thus the assumption depends on the lower degree bound $e$, while the
bounds for the later shifts depend on the upper degree bound $d$.
In particular, if $t_2<2e$, then
\begin{equation}\label{eq:second-intro}
 t_i\leq
 \left\lceil\frac{(2d-1)i}{2}\right\rceil-2(d-e)
 \qquad(2\leq i\leq p).
\end{equation}
For $e=d$, these bounds describe the growth of the defect $di-t_i$.
For $e=d=2$, Theorem~\ref{thm:main} recovers
\cite[Proposition~3.2]{Abed2024}.

The proof is combinatorial.  After polarization, the relevant induced
complexes contain two intersecting minimal nonfaces.  Mayer--Vietoris then
gives
\[
 t_r\leq t_{r-2}+2d-1.
\]
Iterating this inequality proves the theorem.

Ordinary subadditivity
\[
 t_{u+v}(R)\leq t_u(R)+t_v(R)
\]
holds for monomial ideals \cite{ABHMW}.  Our proof is independent of this
result.  Example~\ref{ex:sharp-quadratic} gives a quadratic family for
which our bound on the last shift is smaller than every bound obtained by
repeated use of ordinary subadditivity.
Earlier results on subadditivity for monomial and edge ideals include
\cite{AbedNevo2017,FaridiShahada,JayanthanKumar}.

\section{Deletion covers}

Let $\Delta$ be a simplicial complex on $[n]$, and let $\NN(\Delta)$ be
its set of minimal nonfaces: nonfaces whose proper subsets are faces.
Write
\[
 I_\Delta=(x_F:F\in\NN(\Delta)),\qquad
 x_F=\prod_{v\in F}x_v,\qquad \kk[\Delta]=S/I_\Delta.
\]
For $W\subseteq[n]$, the induced subcomplex is
$\Delta[W]=\{\sigma\in\Delta:\sigma\subseteq W\}$.
Its minimal nonfaces are the members of $\NN(\Delta)$ contained in $W$.
We compute $\pd\kk[\Delta[W]]$ over $\kk[x_v:v\in W]$.

All homology is reduced, with coefficients in $\kk$; we use
$\widetilde H_{-1}(\{\emptyset\})=\kk$ and $\widetilde H_q=0$ for $q<-1$.
Hochster's formula \cite{Hochster} is
\begin{equation}\label{eq:Hochster}
 \beta^S_{i,j}(\kk[\Delta])
 =\sum_{\substack{W\subseteq[n]\\|W|=j}}
 \dim_\kk\widetilde H_{j-i-1}(\Delta[W])
\end{equation}

The next lemma follows from \cite[Lemma~3.1]{Abed2022}.

\begin{lemma}\label{lem:iterated-MV}
Let $\Gamma=\Gamma_1\cup\cdots\cup\Gamma_r$ be a union of subcomplexes.
If $\widetilde H_q(\Gamma)\neq0$, then there is a nonempty
$J\subseteq[r]$ such that
\[
 \widetilde H_{q-|J|+1}
 \left(\bigcap_{j\in J}\Gamma_j\right)\neq0.
\]
\end{lemma}

Let $\Gamma$ be a simplicial complex on $W$, and let
$E_1,\ldots,E_s\subseteq W$ be pairwise disjoint nonempty sets.  Put
\begin{equation}\label{eq:D-definition}
 \Dcal_\Gamma(E_1,\ldots,E_s)
 =\bigcup_{v_1\in E_1,\ldots,v_s\in E_s}
 \Gamma[W\setminus\{v_1,\ldots,v_s\}].
\end{equation}
Its faces omit a vertex from each $E_i$

\begin{lemma}\label{lem:grouped-deletions}
If $\widetilde H_q(\Dcal_\Gamma(E_1,\ldots,E_s))\neq0$, then there are
nonempty $Q_i\subseteq E_i$ such that, for
$Q=Q_1\cup\cdots\cup Q_s$,
\begin{equation}\label{eq:grouped-conclusion}
 \widetilde H_{q-|Q|+s}(\Gamma[W\setminus Q])\neq0.
\end{equation}
\end{lemma}

\begin{proof}
We use induction on $s$.

Suppose first that $s=1$.  By definition,
\[
 \Dcal_\Gamma(E_1)
 =
 \bigcup_{v\in E_1}\Gamma[W\setminus\{v\}].
\]
Apply Lemma~\ref{lem:iterated-MV} to this cover.  Since
$\widetilde H_q(\Dcal_\Gamma(E_1))\neq0$, there is a nonempty
$Q_1\subseteq E_1$ such that
\[
 \widetilde H_{q-|Q_1|+1}
 \left(
 \bigcap_{v\in Q_1}\Gamma[W\setminus\{v\}]
 \right)\neq0.
\]
But
\[
 \bigcap_{v\in Q_1}\Gamma[W\setminus\{v\}]
 =
 \Gamma[W\setminus Q_1].
\]
This proves the statement for $s=1$.

Now let $s>1$ and assume the result holds for $s-1$ groups.
For each $v\in E_1$, set
\[
 \Lambda_v=
 \Dcal_{\Gamma[W\setminus\{v\}]}(E_2,\ldots,E_s).
\]
Then
\[
 \Dcal_\Gamma(E_1,\ldots,E_s)
 =
 \bigcup_{v\in E_1}\Lambda_v.
\]
Indeed, a face in $\Dcal_\Gamma(E_1,\ldots,E_s)$ omits some
$v\in E_1$ and also omits at least one vertex from each of
$E_2,\ldots,E_s$.

By Lemma~\ref{lem:iterated-MV}, there is a nonempty
$Q_1\subseteq E_1$ such that
\[
 \widetilde H_{q-|Q_1|+1}
 \left(\bigcap_{v\in Q_1}\Lambda_v\right)\neq0.
\]
Since the sets $E_1,\ldots,E_s$ are pairwise disjoint,
deleting the vertices of $Q_1$ does not change any of
$E_2,\ldots,E_s$.  Hence
\[
 \bigcap_{v\in Q_1}\Lambda_v
 =
 \Dcal_{\Gamma[W\setminus Q_1]}(E_2,\ldots,E_s).
\]
Therefore
\[
 \widetilde H_{q-|Q_1|+1}
 \bigl(
 \Dcal_{\Gamma[W\setminus Q_1]}(E_2,\ldots,E_s)
 \bigr)\neq0.
\]

Apply the induction hypothesis to
$\Gamma[W\setminus Q_1]$ and the $s-1$ sets
$E_2,\ldots,E_s$.  There are nonempty
$Q_i\subseteq E_i$ for $2\leq i\leq s$ such that, if
\[
 Q'=Q_2\cup\cdots\cup Q_s,
\]
then
\[
 \widetilde H_{
 q-|Q_1|+1-|Q'|+(s-1)}
 \bigl(
 \Gamma[(W\setminus Q_1)\setminus Q']
 \bigr)\neq0.
\]
Since the $Q_i$ are disjoint,
\[
 |Q|=|Q_1|+|Q'|,
 \qquad
 Q=Q_1\cup Q',
\]
and
\[
 (W\setminus Q_1)\setminus Q'=W\setminus Q.
\]
Thus the last nonzero homology group is
\[
 \widetilde H_{q-|Q|+s}(\Gamma[W\setminus Q])
\]
as required.
\end{proof}

\section{The squarefree case}

We use the following bound from \cite[Corollary~4]{HS}.

\begin{lemma}\label{lem:one-step}
Let $R=\kk[\Delta]$.  If every minimal nonface of $\Delta$ has size at
most $d$, then
\begin{equation}\label{eq:one-step}
 t_r(R)\leq t_{r-1}(R)+d\qquad(1\leq r\leq\pd R).
\end{equation}
\end{lemma}

\begin{lemma}\label{lem:intersecting-nonfaces}
Suppose that every minimal nonface of $\Delta$ has size at least $e\geq2$
and that $t_a(\kk[\Delta])<ea$ for some $2\leq a\leq\pd\kk[\Delta]$.
If $W\subseteq[n]$ and $\pd\kk[\Delta[W]]\geq a$, then $\Delta[W]$ has
two distinct intersecting minimal nonfaces.
\end{lemma}

\begin{proof}
Set
\[
 \Gamma=\Delta[W].
\]
Let $F_1,\ldots,F_m$ be the minimal nonfaces of $\Gamma$.
The Stanley--Reisner ideal of $\Gamma$ has $m$ minimal generators, so its
Taylor resolution has length $m$.  Hence
\[
 \pd\kk[\Gamma]\leq m.
\]
Since $\pd\kk[\Gamma]\geq a$, we have $m\geq a$.

Assume that the minimal nonfaces of $\Gamma$ are pairwise disjoint.
Choose $a$ of them, say $F_1,\ldots,F_a$, and set
\[
 U=F_1\cup\cdots\cup F_a.
\]
Because all minimal nonfaces of $\Gamma$ are pairwise disjoint, no other
minimal nonface of $\Gamma$ is contained in $U$.  Therefore
\[
 I_{\Delta[U]}
 =
 I_{\Gamma[U]}
 =
 (x_{F_1},\ldots,x_{F_a}).
\]
The sets $F_1,\ldots,F_a$ are disjoint, so the monomials
$x_{F_1},\ldots,x_{F_a}$ have disjoint supports.  Hence they form a
regular sequence.

Thus $\kk[\Delta[U]]$ is a complete intersection.  Its Koszul resolution
has a nonzero last free module
\[
 S_U(-|U|),
 \qquad
 |U|=|F_1|+\cdots+|F_a|,
\]
where $S_U=\kk[x_i:i\in U]$.  In particular,
\[
 \beta^{S_U}_{a,|U|}(\kk[\Delta[U]])\neq0.
\]
Since $U$ is the full vertex set of $\Delta[U]$, Hochster's formula gives
\[
 \widetilde H_{|U|-a-1}(\Delta[U])\neq0.
\]
Applying Hochster's formula again, now to $\Delta$, gives
\[
 \beta^S_{a,|U|}(\kk[\Delta])\neq0.
\]
Therefore
\[
 t_a(\kk[\Delta])\geq |U|.
\]
Finally, every $F_i$ has size at least $e$, so
\[
 |U|=\sum_{i=1}^a |F_i|\geq ae.
\]
Hence
\[
 t_a(\kk[\Delta])\geq ae,
\]
contrary to the assumption $t_a(\kk[\Delta])<ae$.

Thus $\Gamma=\Delta[W]$ must contain two distinct intersecting minimal
nonfaces.
\end{proof}

\begin{proposition}\label{prop:recurrence}
Let $R=\kk[\Delta]$ and $p=\pd R$.  Suppose that the minimal nonfaces
of $\Delta$ have sizes between $e$ and $d$, where $2\leq e\leq d$,
and that $t_a(R)<ea$ for some $2\leq a\leq p$.  Then, for every
$a+2\leq r\leq p$,
\begin{equation}\label{eq:recurrence}
 t_r(R)\leq
 \max\left\{
 t_{r-1}(R)+d-1,\,
 t_{r-2}(R)+2d-1
 \right\}.
\end{equation}
In particular,
\begin{equation}\label{eq:two-step}
 t_r(R)\leq t_{r-2}(R)+2d-1.
\end{equation}
\end{proposition}

\begin{proof}
Choose $W\subseteq[n]$ with $|W|=t_r(R)$ such that
\[
 \widetilde H_j(\Gamma)\neq0,
 \qquad
 \Gamma=\Delta[W],
 \qquad
 j=|W|-r-1.
\]
By Hochster's formula,
\[
 \beta_{r,|W|}(R)\neq0
 \qquad\text{and}\qquad
 \pd\kk[\Gamma]\geq r\geq a.
\]
Hence Lemma~\ref{lem:intersecting-nonfaces} gives two distinct
intersecting minimal nonfaces $F,G$ of $\Gamma$.

Set
\[
 C=F\cap G,\qquad
 A=F\setminus G,\qquad
 B=G\setminus F.
\]
Since $F$ and $G$ are distinct minimal nonfaces, $A,B,C$ are nonempty.
Also,
\[
 |C|\leq d-1,\qquad
 |A|+|B|\leq2d-2,\qquad
 |A|+|B|+|C|\leq2d-1.
\]
Indeed,
\[
 |A|+|B|+|C|
 =|F\cup G|
 =|F|+|G|-|C|
 \leq2d-1.
\]

Set
\[
 \Gamma_1=\Dcal_\Gamma(C),
 \qquad
 \Gamma_2=\Dcal_\Gamma(A,B).
\]
We have
\[
 \Gamma=\Gamma_1\cup\Gamma_2.
\]
Indeed, if $\sigma\in\Gamma$ and $C\nsubseteq\sigma$, then
$\sigma\in\Gamma_1$.  If $C\subseteq\sigma$, then
$A\nsubseteq\sigma$ and $B\nsubseteq\sigma$, since
$F=A\cup C$ and $G=B\cup C$ are nonfaces.  Hence
$\sigma\in\Gamma_2$.  Also,
\[
 \Gamma_1\cap\Gamma_2=\Dcal_\Gamma(C,A,B).
\]

Since $\widetilde H_j(\Gamma)\neq0$, Mayer--Vietoris shows that at least
one of
\[
 \widetilde H_j(\Gamma_1),\qquad
 \widetilde H_j(\Gamma_2),\qquad
 \widetilde H_{j-1}(\Gamma_1\cap\Gamma_2)
\]
is nonzero.

If $\widetilde H_j(\Gamma_1)\neq0$, then
Lemma~\ref{lem:grouped-deletions} gives a nonempty $Q\subseteq C$ such that
\[
 \widetilde H_{j-|Q|+1}
 \bigl(\Delta[W\setminus Q]\bigr)\neq0.
\]
Since
\[
 (|W|-|Q|)-(j-|Q|+1)-1=r-1,
\]
Hochster's formula gives
\[
 \beta_{r-1,\,|W|-|Q|}(R)\neq0.
\]
Thus
\[
 t_r(R)=|W|
 \leq t_{r-1}(R)+|Q|
 \leq t_{r-1}(R)+d-1.
\]

If $\widetilde H_j(\Gamma_2)\neq0$, then there are nonempty
$Q_A\subseteq A$ and $Q_B\subseteq B$.  For
$Q=Q_A\cup Q_B$,
\[
 \widetilde H_{j-|Q|+2}
 \bigl(\Delta[W\setminus Q]\bigr)\neq0.
\]
Now
\[
 (|W|-|Q|)-(j-|Q|+2)-1=r-2,
\]
so
\[
 \beta_{r-2,\,|W|-|Q|}(R)\neq0.
\]
Therefore
\[
 t_r(R)
 \leq t_{r-2}(R)+|Q|
 \leq t_{r-2}(R)+2d-2.
\]

Finally, suppose that
\[
 \widetilde H_{j-1}(\Gamma_1\cap\Gamma_2)\neq0.
\]
Since
\[
 \Gamma_1\cap\Gamma_2=\Dcal_\Gamma(C,A,B),
\]
Lemma~\ref{lem:grouped-deletions} gives a set
$Q\subseteq A\cup B\cup C$, meeting each of $A,B,C$, such that
\[
 \widetilde H_{j-|Q|+2}
 \bigl(\Delta[W\setminus Q]\bigr)\neq0.
\]
The corresponding homological degree is again $r-2$.  Hence
\[
 t_r(R)
 \leq t_{r-2}(R)+|Q|
 \leq t_{r-2}(R)+2d-1.
\]

Thus
\[
 t_r(R)\leq
 \max\{t_{r-1}(R)+d-1,\ t_{r-2}(R)+2d-1\},
\]
which proves \eqref{eq:recurrence}.

Finally, Lemma~\ref{lem:one-step} gives
\[
 t_{r-1}(R)\leq t_{r-2}(R)+d.
\]
Hence
\[
 t_{r-1}(R)+d-1\leq t_{r-2}(R)+2d-1,
\]
and \eqref{eq:two-step} follows.
\end{proof}

\begin{proof}[Proof of Theorem~\ref{thm:main}]
By polarization, we may assume that $I$ is squarefree.

Let $b\geq0$ and suppose that $a+b\leq p$.

If $b=2q$, then repeated use of \eqref{eq:two-step} gives
\[
 t_{a+2q}(R)\leq t_a(R)+q(2d-1)
 =
 t_a(R)+
 \left\lceil\frac{(2d-1)b}{2}\right\rceil.
\]

If $b=2q+1$, then
\[
 t_{a+2q}(R)\leq t_a(R)+q(2d-1),
\]
and Lemma~\ref{lem:one-step} gives
\[
 \begin{aligned}
 t_{a+2q+1}(R)
 &\leq t_{a+2q}(R)+d\\
 &\leq t_a(R)+q(2d-1)+d\\
 &=
 t_a(R)+
 \left\lceil\frac{(2d-1)b}{2}\right\rceil.
 \end{aligned}
\]

Thus \eqref{eq:main-intro} holds for every $b\geq0$ with
$a+b\leq p$.
\end{proof}

\section{Consequences}

Throughout this section, let $2\leq e\leq d$ be integers and let
$I\subseteq S$ be a monomial ideal with minimal generator degrees
between $e$ and $d$.
Put $R=S/I$, $p=\pd_S R$, $t_i=t_i(R)$, and $\delta_i=di-t_i$.
Thus $t_a<ae$ is equivalent to $\delta_a>a(d-e)$.

\begin{corollary}\label{cor:defect}
If $t_a<ae$ for some $2\leq a\leq p$, then, for $0\leq b\leq p-a$,
\begin{equation}\label{eq:defect-growth}
 \delta_{a+b}\geq\delta_a+\left\lfloor\frac b2\right\rfloor
 \geq a(d-e)+1+\left\lfloor\frac b2\right\rfloor.
\end{equation}
In particular,
\begin{equation}\label{eq:mixed-shift-bound}
 t_{a+b}\leq ae+db-1-\left\lfloor\frac b2\right\rfloor.
\end{equation}
\end{corollary}

\begin{proof}
By Theorem~\ref{thm:main},
\[
 t_{a+b}\leq t_a+
 \left\lceil\frac{(2d-1)b}{2}\right\rceil.
\]
Since
\[
 \left\lceil\frac{(2d-1)b}{2}\right\rceil
 =db-\left\lfloor\frac b2\right\rfloor,
\]
we obtain
\[
 \begin{aligned}
 \delta_{a+b}
 &=d(a+b)-t_{a+b}\\
 &\geq da-t_a+\left\lfloor\frac b2\right\rfloor\\
 &=\delta_a+\left\lfloor\frac b2\right\rfloor.
 \end{aligned}
\]
Moreover, $t_a<ae$, hence
\[
 \delta_a=da-t_a
 \geq da-ae+1
 =a(d-e)+1.
\]
This proves \eqref{eq:defect-growth}.

Finally,
\[
 \begin{aligned}
 t_{a+b}
 &=d(a+b)-\delta_{a+b}\\
 &\leq d(a+b)-a(d-e)-1-\left\lfloor\frac b2\right\rfloor\\
 &=ae+db-1-\left\lfloor\frac b2\right\rfloor,
 \end{aligned}
\]
which gives \eqref{eq:mixed-shift-bound}.
\end{proof}

\begin{corollary}\label{cor:second-shift}
If $p\geq2$ and $t_2<2e$, then
\begin{equation}\label{eq:second-shift-bound}
 t_i\leq
 \left\lceil\frac{(2d-1)i}{2}\right\rceil-2(d-e)
 \qquad(2\leq i\leq p).
\end{equation}
\end{corollary}

\begin{proof}
Take $a=2$ and $b=i-2$ in \eqref{eq:mixed-shift-bound}.  Then
\[
 t_i\leq
 2e+d(i-2)-1-\left\lfloor\frac{i-2}{2}\right\rfloor.
\]
Since
\[
 \left\lfloor\frac{i-2}{2}\right\rfloor
 =\left\lfloor\frac i2\right\rfloor-1,
\]
this becomes
\[
 \begin{aligned}
 t_i
 &\leq di-2(d-e)-\left\lfloor\frac i2\right\rfloor\\
 &=\left\lceil\frac{(2d-1)i}{2}\right\rceil-2(d-e).
 \end{aligned}
\]
\end{proof}

\Needspace{8\baselineskip}
Recall that
\[
 \reg R=\max_{0\leq i\leq p}\{t_i-i\}
\]

\begin{corollary}\label{cor:regularity}
If $t_a<ae$ for some $2\leq a\leq p$, then
\begin{equation}\label{eq:regularity-refined}
 \reg R\leq\max\left\{
 (d-1)(a-1),\ 
 (d-1)p-\delta_a-\left\lfloor\frac{p-a}{2}\right\rfloor
 \right\}.
\end{equation}
In particular,
\begin{equation}\label{eq:regularity-simple}
 \reg R\leq\max\left\{
 (d-1)(a-1),\
 (d-1)p-a(d-e)-1-\left\lfloor\frac{p-a}{2}\right\rfloor
 \right\}.
\end{equation}
\end{corollary}

\begin{proof}
First let $0\leq i<a$.  The Taylor bound gives
\[
 t_i\leq di,
\]
and therefore
\[
 t_i-i\leq(d-1)i\leq(d-1)(a-1).
\]

Now let $i=a+b$, where $0\leq b\leq p-a$.
By Corollary~\ref{cor:defect},
\[
 \delta_{a+b}\geq
 \delta_a+\left\lfloor\frac b2\right\rfloor.
\]
Hence
\[
 \begin{aligned}
 t_i-i
 &=d(a+b)-\delta_{a+b}-(a+b)\\
 &\leq
 (d-1)(a+b)-\delta_a-\left\lfloor\frac b2\right\rfloor.
 \end{aligned}
\]
The right-hand side is nondecreasing in $b$.  Indeed, when $b$ increases
by one, it increases by either $d-1$ or $d-2$, both nonnegative.
Thus its largest value occurs at $b=p-a$, and so
\[
 t_i-i\leq
 (d-1)p-\delta_a-\left\lfloor\frac{p-a}{2}\right\rfloor
\]
for every $a\leq i\leq p$.

Combining the cases $i<a$ and $i\geq a$ gives
\eqref{eq:regularity-refined}.
Finally,
\[
 \delta_a\geq a(d-e)+1
\]
by Corollary~\ref{cor:defect}, and substitution gives
\eqref{eq:regularity-simple}.
\end{proof}

Theorem~\ref{thm:main} and ordinary subadditivity \cite{ABHMW} give the
following bound.

\begin{corollary}\label{cor:min-bound}
If $t_a<ae$ for some $2\leq a\leq p$, then
\begin{equation}\label{eq:min-bound}
 t_{a+b}\leq t_a+
 \min\left\{
 t_b,\ 
 \left\lceil\frac{(2d-1)b}{2}\right\rceil
 \right\}
 \qquad(0\leq b\leq p-a).
\end{equation}
\end{corollary}

\begin{proof}
For $b=0$, the statement is clear.  Let $b\geq1$.
Ordinary subadditivity gives
\[
 t_{a+b}\leq t_a+t_b,
\]
while Theorem~\ref{thm:main} gives
\[
 t_{a+b}\leq
 t_a+\left\lceil\frac{(2d-1)b}{2}\right\rceil.
\]
Taking the minimum gives \eqref{eq:min-bound}.
\end{proof}

\section{Examples and sharpness}

\begin{example}\label{ex:mixed-degree}
Let
\[
 I=x_1x_2(x_3,x_4,x_5x_6,x_5x_7)
 \subseteq \kk[x_1,\ldots,x_7].
\]
The minimal generator degrees are $3,3,4,4$, so $e=3$ and $d=4$.

Set
\[
 J=(x_3,x_4,x_5x_6,x_5x_7).
\]
Since the variable sets are disjoint,
\[
 S/J\cong
 \frac{\kk[x_3,x_4]}{(x_3,x_4)}
 \otimes_\kk
 \frac{\kk[x_5,x_6,x_7]}{(x_5x_6,x_5x_7)}
\]
up to adjoining the variables $x_1,x_2$.
The two factors have maximal shifts $(0,1,2)$ and $(0,2,3)$.
Thus
\[
 (t_0(S/J),\ldots,t_4(S/J))=(0,2,3,4,5).
\]
Since $I=x_1x_2J$, multiplication by $x_1x_2$ shifts the positive
Betti degrees by $2$.  Hence
\[
 (t_1(S/I),t_2(S/I),t_3(S/I),t_4(S/I))
 =(4,5,6,7).
\]

In particular,
\[
 t_2(S/I)=5<6=2e.
\]
Corollary~\ref{cor:second-shift} gives
\[
 t_i(S/I)\leq
 \left\lceil\frac{7i}{2}\right\rceil-2
 \qquad(2\leq i\leq4),
\]
whereas the Taylor bound is
\[
 t_i(S/I)\leq4i.
\]
\end{example}

\begin{example}\label{ex:sharp-quadratic}
Let
\[
 G_0=C_5\sqcup P_3,
\]
where
\[
 E(C_5)=\{12,23,34,45,15\},
 \qquad
 E(P_3)=\{67,78\}.
\]
Let
\[
 R_0=S_0/I(G_0),
 \qquad
 S_0=\kk[x_1,\ldots,x_8].
\]

The maximal shifts of the edge-ideal quotients of $C_5$ and $P_3$ are
\[
 (0,2,3,5)
 \qquad\text{and}\qquad
 (0,2,3),
\]
respectively.
Since the two graphs have disjoint vertex sets, tensoring their minimal
resolutions gives
\begin{equation}\label{eq:base-quadratic-shifts}
 (t_0(R_0),\ldots,t_5(R_0))
 =(0,2,4,5,7,8).
\end{equation}

For $s\geq0$, let $M_s$ be a matching of $s$ edges on new vertices,
and set
\[
 G_s=G_0\sqcup M_s,
 \qquad
 R_s=S_s/I(G_s).
\]
The edge ideal of $M_s$ is a quadratic complete intersection.  Hence
tensoring minimal resolutions gives
\[
 \pd_{S_s}R_s=s+5
\]
and
\[
 t_i(R_s)=
 \max_{\substack{0\leq k\leq5\\0\leq i-k\leq s}}
 \{t_k(R_0)+2(i-k)\}.
\]
Since
\[
 (t_k(R_0)-2k)_{k=0}^5=(0,0,0,-1,-1,-2),
\]
we obtain
\begin{equation}\label{eq:quadratic-family-shifts}
 t_i(R_s)=
 \begin{cases}
  2i,   &0\leq i\leq s+2,\\
  2i-1, &i=s+3,s+4,\\
  2i-2, &i=s+5.
 \end{cases}
\end{equation}

Set
\[
 a=s+3,
 \qquad
 N=s+5.
\]
Then
\[
 t_a(R_s)=2a-1<2a
\]
and
\[
 t_N(R_s)
 =2N-2
 =t_a(R_s)+3.
\]
Thus the constant $3=2d-1$ in the two-step bound is sharp for $d=2$.

Finally, let
\[
 N=i_1+\cdots+i_\ell,
 \qquad
 i_j\geq1,\quad \ell\geq2.
\]
For $i<N$, the only positive defects
\[
 2i-t_i(R_s)
\]
occur at $i=s+3$ and $i=s+4$, and both are equal to $1$.
Since
\[
 2(s+3)>N,
\]
at most one part $i_j$ has positive defect.  Therefore
\[
 \sum_{j=1}^{\ell}t_{i_j}(R_s)\geq2N-1.
\]
On the other hand,
\[
 t_N(R_s)=2N-2.
\]
Hence our bound for the last shift is smaller than every bound obtained
by repeated use of ordinary subadditivity.
\end{example}

\begin{question}
For $d\geq3$, is the constant $2d-1$ in
\[
 t_r\leq t_{r-2}+2d-1
\]
best possible?
\end{question}

\end{document}